\documentclass[a4paper]{amsart}
\usepackage[new]{old-arrows}
\usepackage[T1]{fontenc}
\usepackage[english]{babel}
\usepackage{kpfonts}

\usepackage{amsmath,amssymb,amsthm,enumerate,xspace}
\usepackage{comment}
\usepackage{accents}
\usepackage[colorlinks=true,citecolor=blue,linkcolor=blue]{hyperref}
\usepackage[all]{xy}
\usepackage{tikz}
\usepackage{tikz-cd}
\usepackage{mathrsfs}
\usepackage{cleveref}
\usepackage{graphicx} 
\usepackage{arydshln} % horizontal dashed line

\numberwithin{equation}{section}

\newtheorem{thm}{Theorem}[section]
\newtheorem{cor}[thm]{Corollary}
\newtheorem{prop}[thm]{Proposition}
\newtheorem{lem}[thm]{Lemma}

\newtheorem*{openproblem*}{Problem}
\newtheorem*{quest*}{Question}
\newtheorem*{problem*}{Problem}
\newtheorem*{claim*}{Claim}

\theoremstyle{definition}

\theoremstyle{remark}
\newtheorem{rem}[thm]{Remark}

\newcommand{\bQ}{\mathbb{Q}}
\newcommand{\bR}{\mathbb{R}}
\newcommand{\bS}{\mathbb{S}}

\newcommand{\cC}{\mathcal{C}}

\newcommand\B{\mathrm{B}}
\newcommand\SO{\mathrm{SO}}
\newcommand\U{\mathrm{U}}
\newcommand\Diff{\mathrm{Diff}}

\newcommand{\hcoker}{/\!\!/}

\newcommand\Id{\mathrm{Id}}

\makeatother
\addtocontents{toc}{\protect\setcounter{tocdepth}{1}}
\makeatletter
\let\c@equation\c@thm
\makeatother
\numberwithin{equation}{section}
\makeatletter
\let\@wraptoccontribs\wraptoccontribs
\makeatother
\title[On relative Gelfand-Fuks cohomology of spheres]{A note on relative Gelfand-Fuks cohomology of Spheres}

\author[]{Nils Prigge}
\address{Stockholms universitet\\ 
Matematiska institutionen\\106 91 Stockholm}
\email{nils.prigge@gmail.com}

\begin{document}
\begin{abstract}
We study the Gelfand-Fuks cohomology of smooth vector fields on $\bS^d$ relative to $\mathrm{SO}(d+1)$ following a method by Haefliger that uses tools from rational homotopy theory. In particular, we show that $H^*(\B\SO(4);\bR)$ injects into the relative Gelfand-Fuks cohomology which corrects a claim by Haefliger. Moreover, for $\bS^3$ the relative Gelfand-Fuks cohomology agrees with the smooth cohomology of $\Diff^+(\bS^3)$ and we provide a computation in low degrees.
\end{abstract}
\maketitle
\section{Introduction}
 Let $M$ be a smooth manifold $M$ and $G$ a Lie group acting smoothly and effectively on $M$. We denote by $\mathcal{L}_{M}$ the topological Lie algebra of smooth vector fields on $M$ and identify the Lie algebra of $G$ as a Lie subalgebra $\mathfrak{g}\subset \mathcal{L}_M$ via the action. In this note, we study the continuous Lie algebra cohomology of $\mathcal{L}_{M}$ relative to $\mathfrak{g}$ as introduced by Gelfand-Fuks \cite{GF69}. More precisely, consider the differential, graded commutative algebra of \emph{continuous} multilinear alternating forms $C^*(\mathcal{L}_{M})$ with values in $\bR$, where the differential of a $k$-form $f$ is defined by 
 \[df(v_0,\hdots,v_k)=\sum_{0\leq i<j\leq k}(-1)^{i+j}f([v_i,v_j],v_0,\hdots,\hat{v}_i,\hdots,\hat{v}_j,\hdots,v_k)\in C^{k+1}(\mathcal{L}_M)\]
 for vector fields $v_i\in \mathcal{L}_M$. We denote by $C^*(\mathcal{L}_M;G)\subset C^*(\mathcal{L}_M)$ the subcomplex of $G$-basic cochains, i.e.\,those differential forms which are $G$-invariant and vanish if one of the $v_i$ belongs to $\mathfrak{g}$. For a compact and connected group $G$ the real cohomology of the classifying space $\B G$ agrees with the algebra of $G$-invariant polynomials with respect to the adjoint action$S(\mathfrak{g}^{\vee})^{G}$ by classical Chern-Weil theory and there is a map $S(\mathfrak{g}^{\vee})^{G}\rightarrow C^*(\mathcal{L}_M;G)$ determined by a $G$-connection in $C^*(\mathcal{L}_M)$ (see \cite[Exp.\ 20]{SC49} and \cite[Sect.\ 4.5]{GS99}) which induces a map on cohomology
 \begin{equation}\label{map}
  H^*(\B G;\bR)\longrightarrow H^*(\mathcal{L}_M;G)
 \end{equation}
 that does not depend on the choice of connection. In our main result, we study \eqref{map} for $M=\bS^d$ with respect to the rotation action of $\SO(d+1)$ and correct a statement of Haefliger \cite{Hae78}.
  \begin{thm}\label{Haefliger} Let $M=\bS^d$ and $G=\SO(d+1)$, if $d=3$ the map \eqref{map} is injective. If $d=2n$ is even the kernel of \eqref{map} consists of all polynomials in the Pontrjagin classes $p_1,\hdots,p_{n}$ of degree $>2d$. 
 \end{thm}
  \begin{rem}
 The strategy for the proof is entirely due to Haefliger and we are merely carrying out the program he outlined in \cite{Hae78}, where he also claims that for $M=\bS^d$ and $G=\SO(d+1)$ the kernel of \eqref{map} consists of the all polynomials in the Pontrjagin classes $p_1,\hdots,p_{\lfloor d/2 \rfloor}$ of degree $>2d$ regardless of the parity of $d$. This contradicts recent results of Nariman \cite{Nar23} and \cite{Pri24I} as we explain below, but since Haefliger does not give detailed computations for the kernel of \eqref{map} we do not know the origin of the mistake in \cite{Hae78}. 
 \end{rem}

In the proof of Theorem \ref{Haefliger} we actually prove the following more general statement.
\begin{thm}\label{MainComputation}
 For $d=3$ there exists a class $\bar{z}_1\in H^4 (\mathcal{L}_{\bS^{3}};\SO(4))$ and there is an injection 
 \[ \bR[e,p_1,\bar{z}_1]/\left(p_1^2-e\bar{z}_i\right) \longrightarrow H^* (\mathcal{L}_{\bS^{3}};\SO(4)). \]
\end{thm}
Theorem \ref{Haefliger} and Theorem \ref{MainComputation} have implications for the \emph{smooth cohomology} $H_{\mathrm{sm}}^*(\Diff_+(\bS^d);\bR)$, which can be defined as the cohomology of the subcomplex consisting of continuous cochains
\[\left(\bigoplus_{\geq 0 }\mathrm{Map}(\Diff_+(\bS^d)^k,\bR),\delta\right)\subset \left(\bigoplus_{k\geq 0}\mathrm{Hom}_{\mathrm{Set}}(\Diff_+(\bS^d)^k,\bR),\delta\right)
\]
of the ordinary bar complex that computes the group cohomology of $\Diff_+(\bS^d)$ with coefficients in $\bR$. By \cite[pg.\ 43]{Hae79} there is a Van-Est type isomorphism 
\[H^*_{\mathrm{sm}}(\Diff_+(\bS^d);\bR)\cong H^* (\mathcal{L}_{\bS^d};\SO(d+1))\]
for $d=1,2,3$ and in general there is a conjecture due to Bott \cite[pg.\ 15]{Bo77} which implies that there is a map
	\begin{equation}\label{VanEst}
	H^*_{\mathrm{sm}}(\Diff_+(\bS^d);\bR)\longrightarrow H^* (\mathcal{L}_{\bS^d};\SO(d+1))
	\end{equation}
so that the following diagram commutes
\begin{equation}\label{Map}
 \begin{tikzcd}[ampersand replacement=\&]
H^*(\B\Diff_+(\bS^d);\bR)\arrow{r}\arrow{d} \& H^*_{\mathrm{sm}}(\Diff_+(\bS^d);\bR) \arrow{r}\arrow{d} \& H^*(\B\Diff^{\delta}_+(\bS^d);\bR)\\
H^*(\B\SO(d+1);\bR)\arrow{r} \& H^* (\mathcal{L}_{\bS^d};\SO(d+1))
 \end{tikzcd}
\end{equation}
Here the composition of the top horizontal maps is the map on cohomology induced by $\B\Diff_+^{\delta}(\bS^d)\to \B\Diff_+(\bS^d)$, where $\B\Diff_+^{\delta}(\bS^d)$ is the classifying space of the group of diffeomorphisms with the discrete topology whose cohomology agrees with the group cohomology. There is a map $\B\Diff_+(\bS^d)\to \B\mathrm{Top}_+(d+1)$ so that one can pull back, if $d=2n+1$ is odd, lifts of the Euler and Pontrjagin classes $\bR[e,p_1,\hdots,p_n]\subset H^*(\B\Diff_+(\bS^{d});\bR)$ which we denote the same. Below we summarize a few results in the literature and consequences from Theorem \ref{Haefliger} about \eqref{Map}.
\begin{itemize}
 \item[(i)] For $d=1$ Haefliger has computed in \cite{Hae73} the relative Gelfand-Fuks cohomology $H^*(\mathcal{L}_{\bS^1};\SO(2))=\bR[e,gv]/(e\cdot gv)$, where $gv$ is the so-called Godbillon-Vey class which is of degree $2$ (in particular, this implies that \eqref{map} is injective for $d=1$). Moreover, Morita has shown in \cite{Mor84} that 
 \[H_{\mathrm{sm}}^*(\Diff_+(\bS^1);\bR)\to H^*(\B\Diff_+^{\delta}(\bS^1);\bR)\]
 is injective, so in particular $\bR[e]\hookrightarrow H^*(\B\Diff_+^{\delta}(\bS^1);\bR)$.
 \item[(ii)] For $d=3$ it follows from the Van-Est isomorphism \eqref{VanEst} and Theorem \ref{MainComputation} that there is an injection 
 \[\bR[e,p_1,\bar{z}_1]/(p_1^2-e\bar{x}_1)\longrightarrow H^*_{\mathrm{sm}}(\Diff_+(\bS^3);\bR).\]
 Combined with Nariman's result that $0\neq p_1^2 \in H^8(\B\Diff_+^{\delta}(\bS^3);\bR)$, it was pointed out by Morita that $\bar{x}_1$ is non-trivial in group cohomology which is interesting as it is a continuously varying cohomology class. 
 
 An extension of Nariman's result in fact shows  
 \[\bR[e,p_1]\hookrightarrow H^*(\B\Diff_+^{\delta}(\bS^3);\bR)\] 
 (see \cite{Pri24I}) which implies Theorem \ref{Haefliger} for $d=3$. This statement and Nariman's previous result contradict Haefliger claim \cite{Hae78} that $p_1^2$ vanishes in $H^*(\mathcal{L}_{\bS^3};\SO(4))$, which was the original motivation for this work.
 \item[(iii)] If Theorem \ref{Haefliger} holds more generally for odd $d$ as the author expects, then Bott's conjecture implies that $\bR[e,p_1,\hdots,p_{n}]\subset H^*(\B\Diff_+(\bS^{2n+1});\bR)$ injects into $H^*_{\mathrm{sm}}(\Diff_+(\bS^{2n+1});\bR)$. The main result of \cite{Nar23} is that \[\bR[e,p_1,\hdots,p_{n}]\to H^*(\B\Diff^{\delta}_+(\bS^{2n+1});\bR)\] is non-trivial on the monomials in the Euler and Pontrjagin classes, which implies that they are non-trivial in $H^*_{\mathrm{sm}}(\Diff_+(\bS^{2n+1});\bR)$.
\end{itemize}
\smallskip
These results beg the question about the map 
\[\bR[e,p_1,\hdots,p_{n}]\longrightarrow H^*(\B\Diff_+^{\delta}(\bS^{2n+1});\bR)\]
and suggest that it is injective.
\section{Haefliger's Strategy}
It was conjectured independently by Bott and Fuks and later proved by Haefliger and Bott-Segal that Gelfand-Fuks cohomology $H^*(\mathcal{L}_M)$ can be computed as the cohomology of the following section space: Let $F_d$ be the restriction of the canonical $\text{U}(d)$-bundle over $\mathrm{B}\text{U}(d)$ to the $2d$-skeleton with respect to the CW decomposition of the complex Grassmannian by Schubert cells. For an $d$-dimensional manifold $M$ the associated bundle
\[\pi_M:\text{Fr}^+(TM)\times_{\text{SO}(d)} F_d\longrightarrow M\]
is a fiber bundle whose total space we denote by $E_M$ and we denote by $\Gamma_M$ its space of sections. 
\begin{thm}[\cite{Hae76,BS77}]
 $\mathcal{C}(\mathcal{L}_M)$ is a cdga model for $\Gamma_M$.
\end{thm}
Later, Haefliger extended this result for relative Gelfand-Fuks cohomology. Namely, for a manifold $M$ with a smooth $G$ action of a compact Lie group, $\pi_M\colon E_M\to M$ is a $G$-equivariant fibre bundle and hence $\Gamma_M$ has a $G$-action by conjugation. 
\begin{thm}[{\cite[Thm 1']{Hae78}}]\label{RelativeGFModel}
The complex of continuous and $G$-basic Chevalley-Eilenberg cochains $\mathcal{C}_{CE}(\mathcal{L}_M;G)$ is a cdga model for $\Gamma_M \hcoker G$.  
\end{thm}

Haefliger then used models from rational homotopy theory for section spaces to compute Gelfand-Fuks cohomology. Moreover, he described how to obtain a rational model for $\Gamma_M\hcoker G\to \B G$ in terms of relative Sullivan models of $E_M\hcoker G$ and $M\hcoker G$ over $\B G$, together with a model for the evaluation map
\[\mathrm{ev}:(M\times \Gamma_M)\hcoker G\longrightarrow E_M\hcoker G.\]
We study this model of $\Gamma_M\hcoker G$ for a sphere together with the action of $\SO(d+1)$ given by rotations using recent results from rational homotopy theory. Denote by $B_d$ the cohomology ring $H^*(\B\SO(d);\bQ)$ and define $B_{d+1}$-algebras
\[ A_d:=\begin{cases}
         B_{d+1}[e]/(e^2-p_{d/2}) & \text{if } d\equiv 0 (2)\\
         (B_{d+1}\otimes \Lambda(s),d=e\partial\slash\partial s) & \text{if } d\equiv 1(2)
        \end{cases}
\]
\begin{lem}
 A model for $\bS^d\hcoker \mathrm{SO}(d+1)\to \B\SO(d+1)$ is given by $B_{d+1}\to A_d$.
\end{lem}
\begin{proof}
 Denote by $N\in \bS^d$ the north pole with isotropy group $\SO(d)$. The inclusion $N\hookrightarrow \bS^d$ is equivariant with respect to the inclusion $\SO(d)\hookrightarrow \SO(d+1)$ and induces a homotopy equivalence $*\hcoker \SO(d)\xrightarrow{\simeq} S^d\hcoker \SO(d+1)$. The statement follows from computing a relative Sullivan model of $\B\SO(d)\to\B\SO(d+1)$ which is determined by the induced map on cohomology.
\end{proof}
Similarly, the inclusion of the fibre $\pi_{\bS^d}^{-1}(N)\cong F_d $ into the total space is equivariant with respect to the inclusion $\SO(d)\hookrightarrow \SO(d+1)$ and it follows from the induced map of fibre sequences
\begin{equation*}
	\begin{tikzcd}[ampersand replacement=\&]
	F_d \arrow{r} \arrow[equal]{d} \& F_d\hcoker \SO(d)\arrow{d}\arrow{r} \&N\hcoker \SO(d)=\B\SO(d)\arrow{d}{\simeq}\\
	F_d \arrow{r} \&E_{\bS^d}\hcoker \SO(d+1) \arrow{r} \&\bS^d\hcoker \SO(d+1)
	\end{tikzcd}
\end{equation*}
that the middle map is an equivalence. Therefore, we can obtain a relative model of $E_{\bS^d}\hcoker \SO(d+1)$ from $F_d\hcoker \SO(d)$. A cdga model of $F_d$ is given in \cite{Hae78} and can be obtained by pulling back the model of $\mathrm{E}\U(d)\to \B\U(d)$ given by the Koszul resolution 
\[\bQ[c_1,\hdots,c_d]\longrightarrow (\bQ[c_1,\hdots,c_d]\otimes \Lambda(h_1,\hdots,h_d),d(h_i)=c_i)\]
along the inclusion of $\mathrm{sk}_{2d}\B\U(d)$ which gives
\begin{equation}\label{WU_d}
WU_d:= \left(\frac{\bQ[c_1,\hdots ,c_d]}{(f,\,|f|>2d)}\otimes \Lambda(h_1,\hdots,h_d),d(h_i)=c_i\right).
\end{equation}
From an analogous argument, we obtain a relative model $F_d\hcoker \SO(d)$ which is also stated in \cite[pg.\ 153]{Hae78}.
\begin{lem}\label{ModelFn}
	A relative cdga model of $F_d\hcoker \SO(d)$ is given by 
	\[\left(B_d\otimes WU_d,\tilde{d}(h_i)=c_i-(-1)^{i/2} p_{i/2}\right).\]
\end{lem}
\begin{proof}
	By definition, $\SO(d)$ acts freely on $F_d$ and hence $F_d\hcoker \SO(d)\simeq F_d/\SO(d)$. The projection $F_d/\SO(d)\rightarrow F_d/\U(d)=\text{sk}_{2n}\mathrm{B}\U(d)$ is a $\U(d)/\SO(d)$-bundle and pulled back from $\U(d)/\SO(d)\hookrightarrow\mathrm{B}\SO(d)\rightarrow \mathrm{B}\U(d)$ via the inclusion of the $2n$-skeleton. Rationally, both $\mathrm{B}\SO(d)$ and $ \mathrm{B}\U(d)$ are products of Eilenberg-MacLane whose minimal models coincide with their cohomology rings. Hence, a relative Sullivan model of $i:\mathrm{B}\SO(d)\rightarrow \mathrm{B}\U(d)$ is determined from the induced map on cohomology and given by $E_d:=(B_d\otimes H^*(\mathrm{B}\U(d))\otimes \Lambda(h_1,\hdots,h_d),d(h_k)=c_k-(-1)^{k/2} p_{k/2})$ as $i^*(c_{k})=(-1)^{k/2} p_{k/2}$ if $k$ is even and zero if $k$ is odd. A model for the pullback over the $2d$-skeleton, i.e.\,$F_d/\SO(d)$, is determined by a cdga model of the inclusion $\text{sk}_{2d}\mathrm{B}\U(d)\hookrightarrow \mathrm{B}\U(d)$ via base change by \cite[Prop.\,15.8]{FHT}. Since the inclusion is $(2d+1)$-connected (as there are no cells of odd degree), a minimal model of $\text{sk}_{2d}\mathrm{B}\U(d)$ has generators corresponding to the Chern classes and additional generators of degrees $\geq 2d+1$ and therefore there is a quasi-isomorphism to $H^*(\text{sk}_{2d}\mathrm{B}\U(d))$ by projection onto the Chern classes. It follows that the inclusion of the $2d$-skeleton is formal and thus a model of $F_d/\SO(d)$ is given by 
	\begin{align*}
		\frac{\bQ[c_1,\hdots,c_d]}{(f(c_1,\hdots,c_d),\, |f|>2d)}\otimes_{H^*(\mathrm{B}\U(d))}E_d \cong \left(B_d\otimes WU_d,\tilde{d}(h_i)=c_i-(-1)^{i/2} p_{i/2}\right),
	\end{align*}
	which is isomorphic to the model that Haefliger gives in \cite[Sect.\,7]{Hae78}.
\end{proof}
\begin{cor}\label{corEn}
	The fibre bundle $E_{\bS ^d}\rightarrow \bS^d$ is fibrewise rationally equivalent to the trivial fibration $\pi_2:F_d\times \bS^d\rightarrow \bS^d$
\end{cor}
\begin{proof}
	By construction, $E_{\bS^d}$ is the pullback of $F_d\hcoker \SO(d)\rightarrow \mathrm{B}\SO(d)$ along the classifying map of the tangent bundle $\tau_M:\bS^d\rightarrow \mathrm{B}\SO(d)$ which is a formal map. Hence, a relative Sullivan model of $E_{\bS ^n}\rightarrow \bS^d$ is given by 
	\[H^*(\bS^d)\otimes_{B_d}\left(B_d\otimes WU_d,\tilde{d}(h_i)=c_i-(-1)^{i/2} p_{i/2}\right)\]
	by Lemma \ref{ModelFn}. This cdga is isomorphic to $H^*(\bS^d)\otimes WU_d$ as the total Pontrjagin class of $\bS^d$ is trivial, which proves the claim.
\end{proof}
We can determine a relative Sullivan model of $F_d\hcoker \SO(d)$ from Proposition \ref{ModelFn} and consequently of $E_{\bS^d}\hcoker \SO(d+1)$. The content of the next proposition is that, given relative Sullivan models of $E_{\bS^d}\hcoker \SO(d+1)$ and $\bS^d\hcoker \SO(d+1)$, there exists a relative Sullivan model of $\Gamma_{\bS^d}\hcoker \SO(d+1)$ so that the evaluation map 
\[(\mathrm{ev}:\bS^d\times \Gamma_{\bS^d})\hcoker \SO(d+1)\longrightarrow E_{\bS^d}\hcoker \SO(d+1)\]
has a particularly simple form which in fact \emph{determines} the relative Sullivan model of $\Gamma_{\bS^d}\hcoker \SO(d+1)$.

\smallskip
We first introduce some terminology. It follows from \cite{GF70} and unpublished work of Vey (see \cite{Go74}) that $F_d$ is $2d$-connected and has the rational homotopy type of a bouquet of spheres, and we denote by $L_d$ a minimal dg Lie model. Let $\mathcal{C}_{CE}(L_d)=(\Lambda V_d,d)$ denote the Chevalley-Eilenberg complex where $V_d=sL_d^{\vee}$ is the graded dual of $L_d$ which is concentrated in degrees $V_d=V_d^{\geq 2n+1}$. By Corollary \ref{corEn} the space of sections is rationally equivalent to $\mathrm{Map}(\bS^d,F_d)$ whose dg Lie model is $H^*(\bS^d)\otimes L_d)$ by \cite{Ber15}. Then Chevalley-Eilenberg cochains are given by $\cC_{CE}(H^*(\bS^d)\otimes L_d)=(\Lambda (V_d\oplus \overline{V}_d),D)$ where $\overline{V}_d^*=V_d^{*-d}$. Sometimes, it will be convenient to pick a basis of $V_d$ so that $\cC_{CE}(L_d)=(\Lambda z_i ,d)$, which also gives a preferred choice of generators of $\cC_{CE}(H^*(\bS^d)\otimes L_d)=(\Lambda (z_i,\bar{z_i}),D)$ where $|\bar{z}_i|=|z_i|-d$.

\begin{prop}\label{ModelEv}
 Given a relative Sullivan model $(B_{d+1}\otimes \cC_{CE}(L_d),D)$ of $E_{\bS^d}\hcoker\SO(d+1) $, there is a relative Sullivan model of $\Gamma_{\bS^d}\hcoker \SO(d+1)\to \B\SO(d+1)$ of the form \[(B_{d+1}\otimes \Lambda(z_i,\bar{z_i}),\bar{D})\] so that the evaluation map is modelled by 
 \begin{align}\label{ModelEvaluation}
 \begin{split}
 \mathrm{ev}: (A_d\otimes \Lambda z_i,D) &\longrightarrow A_d\otimes_{B_{d+1}}(B_{d+1}\otimes \Lambda(z_i,\bar{z_i}),\bar{D})\\
  z_i&\longmapsto 1\otimes z_i+s\otimes \bar{z_i}
 \end{split}
 \end{align}
 \end{prop}
 We need the following technical statement regarding relative Sullivan models for the proof.
 \begin{lem}\label{technical}
 	Let $\Psi_1:(B\otimes \Lambda V, D_V)\rightarrow (B\otimes \Lambda W,D_W)$ be a map of relative Sullivan algebras with $B^0=\bQ$. Suppose $\psi_1:=\Psi_1\otimes_B\bQ $ is homotopic to $\psi_2:(\Lambda V,d_V)\rightarrow (\Lambda W,d_W)$, then $\Psi_1$ is homotopic relative $B$ to a map $\Psi_2$ so that $\Psi_2\otimes_B\bQ=\psi_2$.
 \end{lem}
 \begin{proof}
 	There exists a map $h:\Lambda V\rightarrow \Lambda W$ of degree $-1$ so that $\psi_2-\psi_1=d_Wh+hd_V$ by \cite[Prop.\,12.8]{FHT}. Denote by $h_B:B\otimes \Lambda V\rightarrow B\otimes \Lambda W$ its $B$-linear extension and define $\Psi_2$ via its restriction by $\Psi_2|_V=\Psi_1|_V+D_Wh_B+h_BD_V$. Then $\Psi_2$ is a map of $B$-algebras and $\Psi_2\otimes_B\bQ=\psi_1+d_Wh+hd_V=\psi_2$. Moreover, define $H:(B\otimes \Lambda V, D_V)\rightarrow (B\otimes \Lambda W,D_W)\otimes\Lambda(t,dt)$ by 
 	\[H(v)=\Psi_1(v)+(\Psi_2(v)-\Psi_1(v))t-(-1)^{|v|}h_B(v)dt,
 	\]
 	which is a chain map and a homotopy $\Psi_1\sim \Psi_2$ relative $B$.
 \end{proof}
\begin{proof}[Proof of Proposition \ref{ModelEv}]
 It follows from Corollary \ref{corEn} that $\Gamma_{\bS^d}\simeq_{\bQ}\text{Map}(\bS^d,F_d)$ which is $d$-connected and has a Lie model $H^*(\bS^d)\otimes L_d$ by \cite[Thm 1.5]{Ber15}. In the following, we denote $H^*(\bS^d)=\Lambda(s)/(s^2)$. There exists a relative Sullivan model for $\Gamma_{\bS^d}\hcoker \SO(d+1)$ of the form $(B_{d+1}\otimes \mathcal{C}_{CE}(H^*(\bS^d)\otimes L_d),\overline{D})$ that extends the differential of $\mathcal{C}_{CE}(H^*(\bS^d)\otimes L_d)$. The evaluation map $(\Gamma_{\bS^d}\times \bS^d)\hcoker \SO(d+1)\rightarrow E_{\bS^d}\hcoker \SO(d+1)$ is over $\bS^d\hcoker \SO(d+1)$ and hence modeled by a map
\begin{equation}\label{evaluation}
\Psi: (A_d\otimes \mathcal{C}_{CE}(L_d),D)\longrightarrow A_d\otimes_{B_{d+1}}(B_{d+1}\otimes \mathcal{C}_{CE}(H^*(\bS^d)\otimes L_d),\overline{D})
\end{equation}
over $A_d$, where
\begin{itemize}
 \item[(i)] $A_d$ is the relative Sullivan model of $\bS^d\hcoker \SO(d+1)$ over $B_{d+1}$ defined above;
 \item[(ii)] $ (A_d\otimes \mathcal{C}_{CE}(L_d),D)$ is a relative Sullivan model for $E_{\bS^d}\hcoker \SO(d+1)\simeq F_d/\SO(d)$ from Lemma \ref{ModelFn} (and an analogue of Corollary \ref{corEn} for general $d$) by base change along $B_d\overset{\simeq}{\rightarrow} A_d$.
\end{itemize}
Since $E_{\bS^d}$ is rationally equivalent to a trivial fibration, the evaluation map is equivalent over $\mathrm{B}\SO(d+1)$ to $\text{ev}\times \pi_2:\text{Map}(\bS^d,F_d)\times \bS^d\rightarrow F_d\times \bS^d$. If we denote the generators of $\mathcal{C}_{CE}(H(\bS^d)\otimes L_d)$ by $\{z_i,\bar{z}_i\}$ where $|\bar{z}_i|=|z_i|-d$, then a model of the evaluation map of mapping spaces is given by
\begin{equation}\label{eval}
\psi:\mathcal{C}_{CE}(L_d)\rightarrow \mathcal{C}_{CE}(H(\bS^d)\otimes L_d)\otimes H(\bS^d),\quad z_i\longmapsto z_i\otimes 1+\bar{z}_i\otimes s
\end{equation}
by \cite[Thm 3.11]{Ber20}. Hence, we can assume by Lemma \ref{technical} for $d$ odd that 
\begin{equation}\label{modeleval}
\Psi(z_i)=1\otimes z_i+s\otimes \bar{z}_i+1\otimes a_i+s\otimes b_i
\end{equation}
for some $a_i,b_i\in B_{d+1}^+\otimes \Lambda(z_i,\bar{z}_i)$. The same is true for even $d$ but in order to apply Lemma \ref{technical} we have to use a relative Sullivan model $A'_d:=(B_{d+1}\otimes \Lambda(e,y),d(y)=e^2-p_{d/2})$ instead of $A_d$. Using that the projection map $A'_d\rightarrow A_d$ is a quasi-isomorphism, one can see that we also obtain \eqref{modeleval} in the case $n$ is even.

There is an algebra automorphism of $B_{d+1}\otimes \mathcal{C}_{CE}(H^*(\bS^d)\otimes L_d)$ defined by $f(z_i)=z_i+a_i$ and $f(\bar{z}_i)=\bar{z}_i+b_i$ so that $(A_d\otimes f^{-1})\circ \Psi(z_i)=1\otimes z_i+s\otimes \bar{z}_i$. Hence, we can find a model for the evaluation map \eqref{evaluation} by post-composition with $f^{-1}$ so that $\Psi=A_d\otimes \psi$.
\end{proof}
We use Proposition \ref{ModelEv} to determine the differential $\overline{D}$ in terms of the differential $D$ of the relative Sullivan model $(A_d\otimes \Lambda(z_i),D)$ of $E_{\bS^d}\hcoker \SO(d+1)\to \B\SO(d+1)$. For odd $d$ we define a $B_{d+1}$-linear derivation of $B_{d+1}\otimes \Lambda(z_i,\bar{z}_i)$ of degree $d$ by $\Theta(z_i)=\bar{z}_i$ and $\Theta(\bar{z}_i)=0$. For even $d$ there is a direct sum decomposition of dg modules
\begin{equation}\label{splitting}
A_d\otimes _{B_{d+1}}B_{d+1}\otimes \Lambda(z_i,\bar{z_i})\cong B_{d+1}\otimes \Lambda(z_i,\bar{z_i}) \oplus B_{d+1}\otimes \Lambda(z_i,\bar{z_i})[e],
\end{equation}
and for an element $x\in A_d\otimes _{B_{d+1}}B_{d+1}\otimes \Lambda(z_i,\bar{z_i})$ we denote by $x=x_1+x_e$ the decomposition with respect to this splitting.
\begin{cor}
Let $(B_{d+1}\otimes \cC(L_d),D)$ be a relative Sullivan model of $E_{\bS^d}\hcoker \SO(d+1)$ and $(B_{d+1}\otimes \cC(H^*(\bS^d)\otimes L_d),\overline{D})$ be the relative Sullivan model constructed in Proposition \ref{ModelEv}, then  
\begin{align}\label{Dbar}
\overline{D}(z_i)&=\begin{cases} 
                    D(z_i)- e\bar{z}_i & d \text{ odd}\\
                    \mathrm{ev}(D(z_i))_1 & d \text{ even}
                   \end{cases} 
&
\overline{D}(\bar{z}_i)&=\begin{cases}
                          -\Theta(D(z_i)) & d \text{ odd} \\
                          \mathrm{ev}(D(z_i))_e &  d \text{ even}
                         \end{cases}
 \end{align}
\end{cor}
\begin{proof}
 The even case follows immediately as the model of the evaluation map is a chain map and from the direct sum decomposition \eqref{splitting}. For $d$ odd the differential of $\mathrm{ev}(z_i)=1\otimes z_i+s\otimes \bar{z}_i$ is given by
 \[ 1\otimes (\overline{D}(z_i)+e\bar{z}_i)-s\otimes \overline{D}(\bar{z}_i)\]
 which has to agree with $\mathrm{ev}(D(z_i))$. Since $s^2=0\in A_d$, for every $\chi \in B_{d+1}\otimes \Lambda(z_i)$ we have $\mathrm{ev}(\chi)=1\otimes \chi +s\otimes \Theta(\chi)$. Hence, $D(z_i)=\overline{D}(z_i)+e\bar{z}_i$ and $\Theta(D(z_i))=-\overline{D}(\bar{z}_i)$.
\end{proof}
In practice it is, however, quite difficult to determine the complete relative Sullivan model of $E_{\bS^d}\hcoker \SO(d+1)$ from the small cdga model from Lemma \ref{ModelFn}. In order to prove Theorem \ref{Haefliger} we focus only on a part of the differential. Namely, for a relative Sullivan algebra $(B\otimes \Lambda V,D)$ one can decompose the differential as $D=\sum_{i\geq 0}D_i$, where the $D_i$ are defined by as the composition
\[ V\xrightarrow{D|_{V}}B\otimes \Lambda V\xrightarrow{\Id\otimes \mathrm{pr_i}} B\otimes \Lambda^i V.\]
Then $D_k(B\otimes \Lambda^iV)\subset B\otimes \Lambda^{i+k-1}V$ and hence the part of the differential responsible for the kernel of $H^*(B)\to H^*(B\otimes \Lambda V,D)$ is encoded in $D_0$. 

\medskip 
We begin by studying relative Sullivan models of $F_d\hcoker \U(d)\to \B\U(d)$ which determines the relative Sullivan model of $F_d\hcoker \SO(d)$ by base change along $\B\SO(d)\to \B\U(d)$. Since $\U(d)$ acts freely on $F_d$, the Borel construction $F_d\hcoker \U(d)$ is equivalent to the quotient $F_d\slash \U(d)=\mathrm{sk}_{2d}\B \U(d)$, and the projection $p:F_d\hcoker \U(d)\to \B \U(d)$ is equivalent to the inclusion of the $2d$-skeleton which is formal (see proof of Lemma \ref{ModelFn}). Denote $H^*(\B\U(d))=\bQ[c_1,\dots,c_d]$ by $U_d$, then $\ker(p^*)\subset U_d$ is a monomial ideal with minimal generating set consisting of all monomials $m$ with the property that if another monomial $m'$ of smaller degree divides $m$ then $|m'|\leq 2d$. The degrees of the minimal monomial generating set range from $2d+2$ to $4d$ and we denote it by $\{m^i_j\}_{2d+2\leq i \leq 4n,\ 1\leq j \leq k_i}$ where $|m_j^i|=i$, i.e.\ $m^i_1,\hdots,m^i_{k_i}$ runs through all minimal monomial generators of degree $i$. 
\begin{prop}
	The projection $p:F_d\hcoker \U(d)\to \B \U(d)$ has a relative Sullivan model of the form $(U_d\otimes\Lambda (V_1\oplus V_2),D)$ so that $V_1$ is finite dimensional with basis elements $z^j_i$ for $2d+2\leq i\leq 4d$ and $\ 1\leq j\leq k_i$ so that 
	\begin{equation}\label{C1}
		D(z_i^j)=m_i^j\in U_d
	\end{equation}
	and $V_2$ is of finite type and $D_0|_{V_2}=0$, i.e.\ the following the composition vanishes
	\begin{equation}\label{C2}
		V_2\xrightarrow{D|_{1\otimes V_2}}U_d\otimes \Lambda (V_1\oplus V_2)\xrightarrow{\Id\otimes \pi_0} U_d\otimes \Lambda^0(V_1\oplus V_2)\cong U_d.
	\end{equation}
\end{prop}
\begin{proof}
	We will show inductively that there is a relative Sullivan model of the form $(U_d\otimes \Lambda (V_1\oplus V_2\oplus W),D)$ with $W=W^{\geq k+1}$ and $V_i=V_i^{\leq k}$ so that the differential on $V_1\oplus V_2$ satisfies \eqref{C1} and  \eqref{C2}. 
	
	Since $F_d$ is $2d$-connected, the minimal Sullivan model of $F_d$ is of the form $(\Lambda V,d)$ with $V=V^{\geq 2d+1}$. Let $(U_d\otimes \Lambda V,D)$ be any relative Sullivan model of $F_d\hcoker \U(d)$, then $D|_{V^{2d+1}}=D_0|_{V^{2d+1}}$ for degree reasons and $\mathrm{im}(D_0|_{V^{2d+1}})$ is spanned by the monomials $m^{2d+2}_j$ since they generate $\ker(p^*)$ in degree $2d+2$. Hence, we set $W=V^{>2d+2}$ and $V_2=\ker(D_0|_{V^{2d+1}})$ and $V_1$ the span of preimages $z^{2d+2}_1,\hdots,z^{2d+1}_{k_{2d+2}}$ of the monomial generators of $\ker(p^*)$ of degree $2d+2$, which proves the induction hypothesis for $k=2d+1$.
	
	Let $(U_d\otimes \Lambda (V_1\oplus V_2\oplus W),D)$ be a relative Sullivan model as in the induction hypothesis for some $k\geq 2d+1$. Let $I_{d,k}:=\{m_j^i\}_{2d+2\leq i<k,\ 1\leq j\leq k_i}$ be the ideal generated by monomials in the minimal generating set of $U_d^{>2d}$ of degree less than $k$. Observe that degree $k+2$ elements of $U_d$ can be decomposed as
	\[U_d^{k+2}\cong I^{k+2}_{d,k+1}\oplus \mathrm{span}(m^{k+2}_1,\dots,m^{k+2}_{k_{k+2}}).\]
	For each $m^{k+2}_j$ there exists $v_j\in W^{k+1}$ with $D_0(v_j)=m^{k+2}_j$ and $z_j\in U_d^+\otimes \Lambda (V_1\oplus V_2)$ so that $D(v_j+z_j)=m^{k+2}_j$ because all elements of degree $>2d$ are in the kernel of $p^*:U_d\to H^*(F_d\hcoker \U(d);\bR)$. Hence, 
	\[W^{k+1}=(D_0|_{W^{k+1}})^{-1}(I_{d,k+1}^{k+2})\oplus \mathrm{span}(v_1,\hdots,v_{k_{k+2}})\]
	and we set 
	\begin{align*}
		V_1^{k+1}&:=\mathrm{span}(v_1,\hdots,v_{k_{k+2}}),\\
		V_2^{k+1}&:=(D_0|_{W^{k+1}})^{-1}(I_{d,k+1}^{k+2}),
	\end{align*}
	and $W=W^{\geq k+2}$. In order to satisfy \eqref{C1} and \eqref{C2} we have to modify the differential. Observe that for $x\in V_2^{k+1} $ one can easily construct an element $z_x\in U^+_d\otimes V_1$ so that $D_0(x+z_x)=0$ by using \eqref{C1} from the induction assumption. Picking a basis $w_1^{k+1},\dots,w_{l_{k+1}}^{k+1}$ of $V_2^{k+1}$ we can define an automorphism $\phi_{k+1}$ of $U_d\otimes \Lambda( V_1\oplus V_2\oplus W)$ over $U_d$ which is the identity on generators except in degree $k+1$, where we define $\phi$ on the basis chosen above by
	\begin{align*}
		\phi(w^{k+1}_j)&:=w^{k+1}_j+z_{w^{k+1}_j} & \phi(v^{k+1}_j)&:=v^{k+1}_j+z_{v^{k+1}_j}
	\end{align*}
	The conjugation $\phi^{-1}\circ D\circ \phi$ is a new differential which is unchanged on $(V_1\oplus V_2)^{\leq k}$ and which satisfies \eqref{C1} and \eqref{C2} by construction.
	
	Once $k>4d$ for any element in $W$ with $D_0(w)\neq 0$ we can find an element in $z_w\in U_d^+\otimes \Lambda V_1$ so that $D_0(w+z_w)=0$ and we can repeat the previous construction and modify the differential on generators of degree $>k$ so that $D_0$ vanishes and which doesn't change \eqref{C1} and \eqref{C2} in degree $\leq k$, which is the model we wanted to construct.
\end{proof}
We obtain a model of $F_d\hcoker \SO(d)$ by base change along $\B\SO(d)\to \B\U(d)$ which gives the following statement.
\begin{cor}\label{auxilliary}
There exists a relative Sullivan model of $F_d\hcoker \SO(d)$ of the form 
\[(B_{d}\otimes \Lambda(V_1\oplus V_2),D)\]
extending the differential on $\mathcal{C}_{CE}(L_d)$ so that for a basis of $z_1,\hdots,z_N$ of $V_1$ we have $D(z_i)=f_i$, where $f_1,\dots,f_N\in B_d$ denotes a minimal generating set of $\ker(H^*(\B\SO(d);\bQ)\to H^*(F_d\hcoker \SO(d);\bQ)$, and $D_0|_{V_2}=0$.
\end{cor}

Theorem \ref{Haefliger} and Theorem \ref{MainComputation} follow from the following statement combined with Theorem \ref{RelativeGFModel}.
\begin{thm}\label{MAIN}
The map on cohomology induced by the projection map 
\[p:\Gamma_{\bS^{d}}\hcoker \SO(d+1)\to \B \SO(d+1)\]
is injective for $d=3$. More precisely, there exists a class $\bar{z}_1\in H^4(\Gamma_{\bS^d}\hcoker \SO(d+1);\bR)$ so that 
\begin{align*}
 \bR[e,p_1\bar z_1]/(p_1^2-e\bar{z}_1)\longrightarrow H^*(\Gamma_{\bS^d}\hcoker \SO(d+1);\bR)
\end{align*}
is injective. If $d$ is even the kernel of $p^*$ is generated by all polynomials in the Pontrjagin classes $p_1,\hdots,p_{d/2}$ of degree $>2d$. 
\end{thm}
\begin{proof}
We first prove the claim for $d$ even. First observe that any polynomial in the Pontrjagin classes $p_1,\hdots,p_{d/2}$ of degree $>2d$ vanishes in $H^*((\Gamma_{\bS^d}\times \bS^d)\hcoker \text{SO}(d+1))$. This is because the projection to $\B\SO(d+1)$ factors through
\[
 (\Gamma_{\bS^d}\times \bS^d)\hcoker \SO(d+1))\xrightarrow{\mathrm{ev}} E_{\bS^d}\hcoker \SO(d+1) \to \B\SO(d+1)
\]
and all such elements vanish in $E_{\bS^d}\hcoker \SO(d+1)$ by Corollary \ref{auxilliary}. Moreover, the fibration 
\[\bS^d\longrightarrow (\Gamma_{\bS^d}\times \bS^d)\hcoker \SO(d+1) \xrightarrow{\mathrm{pr}_1} \Gamma_{\bS^d}\hcoker \SO(d+1)\]
satisfies the assumptions of the Leray-Hirsch theorem as the pullback of the Euler class in $\bS^d\hcoker \SO(d+1)\simeq \B\SO(d)$ along $\mathrm{pr}_2$ restricts to a generator $H^n(\bS^d)$ so that $\mathrm{pr}_1$ induces an injection on cohomology. Therefore any element in $H^{>2d}(\B\SO(d+1);\bR)$ vanishes in $H^*(\Gamma_{\bS^d}\hcoker \SO(d+1);\bR)$ already. It remains to show that \eqref{map} is injective in degrees $\leq 2d$. This follows because $|z_i|>2d$ for all $i$ and  $\overline{D}_0(\bar{z}_i)=0$ for all $i$ by Corollary \ref{Dbar} so that no element of degree $<2d$ can be a coboundary of an element in $B_{d+1}^{\leq 2d}$.

\medskip
For $d=2n+1$ odd observe that
\[ (B_{d+1}\otimes \Lambda(V_1,\overline{V}_1),\overline{D}) \subset (B_{d+1}\otimes \mathcal{C}_{CE}(H^*(\bS^d)\otimes L_d),\overline{D})\]
is a subcomplex since 
\begin{align*}
 \overline{D}(z_i)&=m_i-e\bar{z}_i \in B_{d+1}\otimes \Lambda(V_1,\overline{V}_1),\overline{D}),\\
 \overline{D}(\bar{z}_i)&=-\Theta(m_i)=0,
\end{align*}
by Corollary \ref{Dbar}, where $z_i$ is a basis element of $V_1$ with $D(z_i)=m_i$ and the second equality holds as $\Theta$ is $B_{d+1}$-linear. Denote this subcomplex by $C_1$ and the cokernel by $C_2$ and consider the short exact sequences
\begin{equation}\label{sesComplexes}
0 \to C_1\to (B_{d+1}\otimes \mathcal{C}_{CE}(H^*(\bS^d)\otimes L_d),\overline{D}) \to C_2 \to 0. 
\end{equation}
The subcomplex $C_1$ is a sub-cdga which is a pure Sullivan algebra \cite[Ch.\ 32]{FHT}, i.e.\ the differential is only non-trivial for odd degree generators and lies in the algebra of even degree generators. Hence, it has an additional homological grading $F_kC_1:=B_{d+1}\otimes \Lambda \overline{V}_1\otimes \Lambda^k V_1$ and 
\[H_0(C_1)= \bQ[e,p_1,\dots,p_n,\{\bar z_i\}]\slash (\{m_i-e\bar{z}_i\}).\]
The map $B_{d+1}\to (B_{d+1}\otimes \mathcal{C}_{CE}(H^*(\bS^d)\otimes L_d),\overline{D})$ factors through $F_0C_1$ and hence we need to understand the image of the connecting homomorphism of the above short exact sequence. 

For $d=3$ the vector space $V_1$ is $1$-dimensional with generator $z_1$ and $D(z_1)=p_1^2$ by Corollary \ref{auxilliary}. Then the sequence \eqref{sesComplexes} actually splits as a sequence of cochain complexes so that 
\[H(C_1)\hookrightarrow H^*((B_{d+1}\otimes \mathcal{C}_{CE}(H^*(\bS^d)\otimes L_d),\overline{D})).\] 
This is because 
\[ \tilde C_2:=B_{d+1}\otimes \Lambda (V_1\oplus \overline{V}_1)\otimes \Lambda^+(V_2\oplus \overline{V}_2)\]
is a subcomplex isomorphic to $C_2$ as the only way $\overline{D}|_{\tilde C_2}$ can have image in $C_1$ is if there is $w\in V_2$ so that $D(w)=f z_1+\chi$ for $f\in B_{d+1}$ and $\chi \in B_{d+1}\otimes \Lambda V_1\otimes \Lambda^+ V_2$. Suppose $w$ has the minimal degree where this happens, then 
\begin{align*}
 0=D^2(w)=fp_1^2+D(\chi)
\end{align*}
and since $D(\chi)$ cannot be contained in $B_{d+1}\otimes 1$ by construction it follows that $f=0$. Hence, $D^1|_{V_2}$ has image in $B_{d+1}\otimes V_2$ and consequently \eqref{sesComplexes} splits as as cochain complexes. This proves the second part the theorem as 
\[H(C_1)\cong \bQ[e,p_1,\bar{z}_1]/(p_1^2-e\bar z_1)\]
injects into $H^*((B_{d+1}\otimes \mathcal{C}_{CE}(H^*(\bS^d)\otimes L_d),\overline{D}))\cong H^*(\Gamma_{\bS^3}\hcoker \SO(4);\bQ) $. The first part then follows as the map $B_4\to H(C_1)$ is injective.
\end{proof}
\begin{rem}
	The formulation of Corollary \ref{auxilliary} for all dimensions $d$ is due to the author's attempt to prove a general version of Theorem \ref{MAIN}, namely that $H^*(\B\SO(d+1))\to H^*(\Gamma_{\bS^d}\hcoker \SO(d+1)$ is injective for all odd $d$. However, for $d\geq 5$ the free resolution of the kernel of $H^*(\B\SO(d))\to H^*(F_d\hcoker \SO(d))$, i.e.\ the ideal generated by monomials in the Pontrjagin classes of degree $>2d$, has higher syzygies and the proof for $d=3$ does not generalize to show that $D^1(V_2)\subset B_{d+1}\otimes  V_2$. In fact, one would need to show a stronger version of Corollary \ref{auxilliary} to get a splitting of \eqref{sesComplexes} for odd $d>3$. Nonetheless, it follows from Nariman's results that the monomials in the Euler and Pontrjagin classes do not vanish and I expect that one can improve the statement of Corollary \ref{auxilliary} or control the image of the connecting homomorphism of \eqref{sesComplexes} in order to find a purely algebraic proof of Nariman's result. For this reason, I have kept a general version of Corollary \ref{auxilliary} in this article in the hope that it can serve as a starting point to prove a more general version of Theorem \ref{MAIN} -- it is hard to imagine that the statement doesn't generalize to all odd $d$.
\end{rem}

\section{Smooth cohomology of \texorpdfstring{$\Diff_+(\bS^3)$}{Diff(S3)} in low degrees}
The original appendix to \cite{Nar23} that preceded this article contained the computation of a relative Sullivan model of $E_{\bS^d}\hcoker \SO(d+1)$ in low degrees for $d=3$, which is not needed anymore for the main result. However, it does provide a computation of the smooth cohomology up to degree $12$, following the analogous steps of Haefliger's computation of $H^{\leq 9}_{\mathrm{sm}}(\Diff_+(\bS^2))$ in \cite[Sect.\ 7]{Hae78}. As this might be of some independent interest we have included it here.
\begin{prop}\label{SmoothCohomology}
 The non-trivial part of the smooth cohomology of $\Diff_+(\bS^3)$ in degrees $\leq 12$ is given by
 \begin{center}
 \begin{tabular}{c| c c c c c c c c}
  k & $4$ & $6$ & $8$ & $9$ & $10$ & $11$ & $12$ \\ \hline
  $\dim_{\bR} H^k_{\mathrm{sm}}(\Diff_+(\bS^3))$ &  $6$  & $1$ & $15$ &$3$ & $4$ & $3$ & $31$
 \end{tabular}
 \end{center}
\end{prop}
A straightforward computation from \eqref{WU_d} shows that
\begin{equation*}
  F_3\simeq_{\bQ} \bigvee^4S^7\vee S^9\vee \bigvee^3 S^{10} \vee S^{11}\vee\bigvee ^4 S^{12}\vee S^{14}\vee\bigvee ^3S^{15}.
\end{equation*}
Hence, a minimal Quillen model is given by
\[L_3=(\mathbb{L}(s^{-1}\overline{H}_*(F_3)),d=0)\]
and we denote by $x_1,\hdots,x_{17}\in \overline{H}_*(F_3)^{\vee}\subset \cC_{CE}(L_3)$ the algebra generators corresponding to the top degree cohomology of the wedge summands. Generators of $\cC_{CE}^{\leq 18}(L_3)$ are contained in $s(L_3^{\leq 2})^{\vee}\subset sL_3^{\vee}$, where $L_3^{\leq 2}$ denotes the subspace of elements of bracket length $\leq 2$. Since $L_3^2$ is isomorphic to the exterior algebra $\Lambda^2 s^{-1}\overline{H}_*(F_3)$, a basis of $\subset s(L^{2}_3)^{\vee}$ is given by $x_{i,j}$ for $1\leq i\leq j \leq 17$ of degree $|x_i|+|x_j|-1$ (note that if $|x_i|$ is odd, then $x_{i,i}=0$). Picking representatives of a basis of $H^*(WU_3)$ determines a quasi-isomorphism
\[\phi:\cC_{CE}(L_3)\longrightarrow WU_3,\]
which we choose as follows:

\begin{center}
 \begin{tabular}{ c| l }
 degree &   \\\hline
 $7$ &   \begin{minipage}[c]{0.7\textwidth}
       { \begin{align*}
         \phi(x_1)&=c_1h_3, & \phi(x_2)&=c_2h_2, & \phi(x_3)&=c_1^3h_1, & \phi(x_4)&=c_1c_2h_1
        \end{align*}
        }
       \end{minipage}\\
 $9$ & $\phi(x_5)=c_2h_3$\\
 $10$ & \begin{minipage}[c]{0.7\textwidth}
       { \begin{align*}
         \phi(x_6)&=c_2h_1h_3-c_1h_2h_3, & \phi(x_7)&=c_1^3h_1h_2, & \phi(x_8)&=c_1c_2h_1h_2
        \end{align*}
        }
       \end{minipage}\\
 $11$ & $\phi(x_9)=c_3h_3$\\
 $12$ &   \begin{minipage}[c]{0.7\textwidth}
       { \begin{align*}
         \phi(x_{10})&=c_2h_2h_3, & \phi(x_{11})&=c_1^3h_1h_3, & \phi(x_{12})&=c_1c_2h_1h_3, & \phi(x_{13})&=c_3h_1h_3
        \end{align*}
        }
        \end{minipage}\\
 $14$ & $\phi(x_{14})=c_3h_2h_3$\\
 $15$ &  \begin{minipage}[c]{0.7\textwidth}
       { \begin{align*}
         \phi(x_{15})&=c_1^3h_1h_2h_3, & \phi(x_{16})&=c_1c_2h_1h_2h_3, & \phi(x_{17})&=c_3h_1h_2h_3
        \end{align*}
        }
        \end{minipage}
 \end{tabular}
\end{center}
One can check that $\phi(x_i)\cdot \phi(x_j)$ vanishes for all $i,j$ except for $i=1,\ j=2$, where $\phi(x_1)\cdot \phi(x_2)=-c_1c_2h_2h_3=d(-c_2h_1h_2h_3)$. Since $d(x_{1,2})=x_1x_2\in \cC_{CE}(L_3)$ we set 
\[ 
 \phi(x_{i,j})=\begin{cases}
                -c_2h_1h_2h_3 & i=1,\ j=2\\
                0 & \text{otherwise}
               \end{cases}.
\]
It follows from degree reasons that $\phi$ vanishes for generators in $s(L_3^{>2})^{\vee}$, so this completely determines $\phi$. We then use the relative model from Lemma \ref{ModelFn} to find a $B_3$-linear differential $D$ on $B_3\otimes \cC_{CE}(L_3)$ and a quasi-isomorphism 
\[\Psi:(B_3\otimes \cC_{CE}(L_3),D)\longrightarrow (B_3\otimes WU_3,\tilde{d})\]
of $B_3$-algebras that extends $\phi$. 
\begin{lem}
	A relative Sullivan model of $F_3\hcoker \text{SO}(3)$ of the form $(B_3\otimes \mathcal{C}_{CE}(L_3),D)$ extending the differential on $\mathcal{C}_{CE}(L_3)$ for generators of degrees $\leq 15$ is given by 
	\begin{equation}\label{D}
		\begin{aligned}
			D(x_2)&=-p_1^2, & D(x_6)&=-p_1x_1, & D(x_7)&=-p_1x_3,\\
			D(x_8)&=-p_1x_4, & D(x_{10})&=p_1x_5, & D(x_{14})&=p_1x_9,\\
			D(x_{15})&=-p_1x_{11}, & D(x_{16})&=-p_1x_{12}, & D(x_{17})&=-p_1x_{13},\\
			D(x_{1,2})&=x_1x_2+p_1x_6, & D(x_{1,3})&=x_1x_3, & D(x_{1,4})&=x_1x_4,\\
			D(x_{2,3})&=x_2x_3-p_1x_7, & D(x_{2,4})&= x_2x_4- p_1x_8, & D(x_{3,4})&=x_3x_4,\\
			D(x_{1,5})&=x_1x_5, & D(x_{2,5})&=x_2x_5+ p_1x_{10}, & D(x_{3,5})&=x_3x_5,\\
			D(x_{4,5})&=x_4x_5,			
		\end{aligned}
	\end{equation}
	and $D(x_i)=0$ for $i=1,3,4,5,9,11,12,13$. 
\end{lem}
\begin{proof}
Let $\Psi:(B_3\otimes \cC_{CE}(L_3),D)\to (B_3\otimes WU_3,\tilde{d})$ be a relative Sullivan model of $B_3\to (B_3\otimes WU_3,\tilde{d})$. Since $\cC_{CE}(L_3)$ is a minimal Sullivan algebra and quasi-isomorphisms of $\cC_{CE}(L_3)$ are isomorphisms, one can modify any relative Sullivan model so that $\Psi\otimes_{B_3}\bQ=\phi$. We can construct $\Psi$ inductively using the filtration of $s(L_3)^{\vee}$ given by degree or bracket length. Assuming we have defined $D$ and $\Psi$ for generators of filtration $\leq k$, for every generator $z$ of degree $k+1$ we can find elements $a\in B_3^+\otimes WU_3$ and $b\in B_3^+\otimes \cC_{CE}(L_3)$ so that 
\[\tilde{d}(\phi(z)+a)=\Psi(d(z)+b),\]
and we set 
\begin{align*}
 D(z)&:=d(z)+b,\\
 \Psi(z)&:=\phi(z)+a.
\end{align*}
For example, in degree $7$ there are generators $x_i$ for $i=1,2,3,4$. For $i\neq 2$ we have that $\tilde{d}(\phi(x_i))=0$ and for all $i$ the differential $d(x_i)$ vanishes. Hence, we can set $D(x_i)=0$ and $\Psi(x_i)=\phi(x_i)$ for $i=1,3,4$. For $i=2$ we have that $\tilde{d}(\phi(x_2))=p_1c_2$ and therefore
\[\tilde{d}(\phi(x_2)-p_1h_2)=-p_1^2=\Psi(-p_1^2)\]
since $\Psi$ is a $B_3$-algebra homomorphism, and hence 
\begin{align*}
 D(x_2)&=-p_1^2,  & \Psi(x_2)&=\phi(x_2)-p_1h_2.
\end{align*}
It is a straightforward yet lengthy computation to show $\Psi(x_i)=\phi(x_i)$ and differential as in \eqref{D} for all generators $x_i$ of bracket length one except for $i=2$. And similar computations for the next inductive step involving generators of bracket length $2$ and degree $\leq 15$  show that $\Psi(x_{i,j})=\phi(x_{i,j})$ with the differential as stated in \eqref{D}.
\end{proof}
\begin{proof}[Proof of Proposition \ref{SmoothCohomology}]
Using Proposition \ref{ModelEv} we obtain a $B_4$-relative Sullivan model of $\Gamma_{\bS^3}\hcoker \SO(4)$ with algebra generators of degree $\leq 12$ given in the table below. The differential can be computed from Corollary \ref{Dbar} and \eqref{D}, and we give the first few terms below.

\smallskip
\begin{center}
 \begin{tabular}{c|c c}
 degree & generators & differential \\\hline 
 $4$  & $\bar{x}_1,\ \bar{x}_2,\ \bar{x}_3,\ \bar{x}_4$ & $\overline{D}(\bar{x}_i)=0\quad i=1,2,3,4$  \\[5pt]  
 $5$ & - & -\\[5pt]
 $6$ & $\bar{x}_5 $ & $\overline{D}(\bar{x}_5)=0$ \\[5pt] 
 $7$ & $\bar{x}_6,\ \bar{x}_7,\ \bar{x}_8,\ x_1,\ x_2,\ x_3,\ x_4$ & \begin{minipage}[c]{0.33\textwidth}
                                                                     {\begin{align*}
                                                                      \overline{D}(\bar{x}_6)&=p_1\bar{x}_1, \\ \overline{D}(\bar{x}_7)&=-p_1\bar x_3,\\ \overline{D}(\bar{x}_8)&=-p_1\bar x_4, \\
                                                                      \overline{D}(x_1)&=-e\bar{x}_1, \\
                                                                      \overline{D}(x_2)&=-p_1^2-e\bar x_2 \\
                                                                      \overline{D}(x_3)&=-e\bar x_3 \\ \overline{D}(x_4)&=-e\bar x_4 \\
                                                                     \end{align*}}
                                                                     \end{minipage}
 \\[5pt] 
 $8$ & $\bar{x}_9$ & etc. \\[5pt]
 $9$ & $\bar{x}_{10},\ \bar{x}_{11},\ \bar{x}_{12},\ \bar{x}_{13},\ x_5$ & \\[5pt]
 $10$ & $x_6,\ x_7,\ x_8,\ \{\overline{x_{i,j}}\}_{1\leq i<j\leq 4}$ & \\[5pt]
 $11$ & $\bar{x}_{14},\ x_9$ & \\[5pt]\hdashline
 $12$ & \begin{minipage}[c]{0.5\textwidth}
         \begin{align*}\bar{x}_{15},\ \bar{x}_{16},\ \bar{x}_{17}\\
          x_{10},\ x_{11},\ x_{12},\ x_{13}\\
          \overline{x_{1,5}},\ \overline{x_{2,5}},\ \overline{x_{3,5}},\ \overline{x_{4,5}} 
        \end{align*}
        \end{minipage}
      & 
 \end{tabular}
\end{center}
With this it is now a straightforward computation to determine $H^*(\Gamma_{\bS^3}\hcoker \SO(4);\bR)$ in degrees $*\leq 12$. The statement then follows from Theorem \ref{RelativeGFModel} and the van-Est isomorphism $H_{\mathrm{sm}}(\Diff_+(\bS^3))\cong H^*(\mathcal{L}_{\bS^3};\SO(4))$ (see \cite[pg.\ 43]{Hae79}).
\end{proof}

\subsection*{Acknowledgements} I would like to thank Sam Nariman for asking me the question about Haefliger's work that started this project. I also thank Alexander Berglund for helpful discussions. This research was supported by the Knut and Alice Wallenberg foundation through grant no.\ 2019.0519.

\bibliographystyle{alpha}
\bibliography{../../../Bibliography/central-bib}

\begin{thebibliography}{FHT01}

\bibitem[Ber15]{Ber15}
A.~Berglund.
\newblock Rational homotopy theory of mapping spaces via {L}ie theory for
  {$L_\infty$}-algebras.
\newblock {\em Homology Homotopy Appl.}, 17(2):343--369, 2015.

\bibitem[Ber22]{Ber20}
A.~Berglund.
\newblock Characteristic classes for families of bundles.
\newblock {\em Selecta Math. (N.S.)}, 28(3):Paper No. 51, 56, 2022.

\bibitem[Bot77]{Bo77}
Raoul Bott.
\newblock On the characteristic classes of groups of diffeomorphisms.
\newblock {\em Enseign. Math. (2)}, 23(3-4):209--220, 1977.

\bibitem[BS77]{BS77}
R.~Bott and G.~Segal.
\newblock The cohomology of the vector fields on a manifold.
\newblock {\em Topology}, 16(4):285--298, 1977.

\bibitem[FHT01]{FHT}
Y.~F\'elix, S.~Halperin, and J.-C. Thomas.
\newblock {\em Rational homotopy theory}, volume 205 of {\em Graduate Texts in
  Mathematics}.
\newblock Springer-Verlag, New York, 2001.

\bibitem[GF69]{GF69}
I.~M. Gelfand and D.~B. Fuks.
\newblock Cohomologies of the {L}ie algebra of tangent vector fields of a
  smooth manifold.
\newblock {\em Funkcional. Anal. i Prilo\v zen.}, 3(3):32--52, 1969.

\bibitem[GF70]{GF70}
I.~M. Gelfand and D.~B. Fuks.
\newblock Cohomologies of the {L}ie algebra of formal vector fields.
\newblock {\em Izv. Akad. Nauk SSSR Ser. Mat.}, 34:322--337, 1970.

\bibitem[God74]{Go74}
C.~Godbillon.
\newblock Cohomologies d'alg\`ebres de {L}ie de champs de vecteurs formels.
\newblock In {\em S\'eminaire {B}ourbaki, 25\`eme ann\'ee (1972/1973)}, volume
  Vol. 383 of {\em Lecture Notes in Math.}, pages Exp. No. 421, pp. 69--87.
  Springer, Berlin-New York, 1974.

\bibitem[GS99]{GS99}
V.~Guillemin and S.~Sternberg.
\newblock {\em Supersymmetry and equivariant de {R}ham theory}.
\newblock Mathematics Past and Present. Springer-Verlag, Berlin, 1999.
\newblock With an appendix containing two reprints by Henri Cartan [MR0042426
  (13,107e); MR0042427 (13,107f)].

\bibitem[Hae73]{Hae73}
A.~Haefliger.
\newblock Sur les classes caract\'eristiques des feuilletages.
\newblock In {\em S\'eminaire {B}ourbaki, 24\`eme ann\'ee (1971/1972)}, volume
  Vol. 317 of {\em Lecture Notes in Math.}, pages Exp. No. 412, pp. 239--260.
  Springer, Berlin-New York, 1973.

\bibitem[Hae76]{Hae76}
A.~Haefliger.
\newblock Sur la cohomologie de l'alg\`ebre de {L}ie des champs de vecteurs.
\newblock {\em Ann. Sci. \'Ecole Norm. Sup. (4)}, 9(4):503--532, 1976.

\bibitem[Hae78]{Hae78}
A.~Haefliger.
\newblock On the {G}elfand-{F}uks cohomology.
\newblock {\em Enseign. Math. (2)}, 24(1-2):143--160, 1978.

\bibitem[Hae79]{Hae79}
A.~Haefliger.
\newblock Differential cohomology.
\newblock In {\em Differential topology ({V}arenna, 1976)}, pages 19--70.
  Liguori, Naples, 1979.

\bibitem[Mor84]{Mor84}
S.~Morita.
\newblock Nontriviality of the {G}elfand-{F}uchs characteristic classes for
  flat {$S\sp{1}$}-bundles.
\newblock {\em Osaka J. Math.}, 21(3):545--563, 1984.

\bibitem[Nar23]{Nar23}
S.~Nariman.
\newblock On invariants of foliated sphere bundles.
\newblock {\em to appear in Commentarii Mathematici Helvetici}, 2023.

\bibitem[Pri24]{Pri24I}
N.~Prigge.
\newblock A note on invariants of foliated 3-sphere bundles.
\newblock {\em arXiv preprint arXiv:2409.06408}, 2024.

\bibitem[SC456]{SC49}
{\em S\'eminaire {H}enri {C}artan de l'{E}cole {N}ormale {S}up\'erieure,
  1949/1950. {E}spaces fibr\'es et homotopie}.
\newblock Secr\'etariat Math\'ematique, 11 rue Pierre Curie, Paris, 1956.
\newblock 18 expos\'es par Blanchard, A.; Borel, A.; Cartan, H.; Serre, J. P.;
  and Wen Ts\"un, Wu, 2\`eme \'ed.

\end{thebibliography}
\end{document}